\documentclass[12pt]{amsart}
\font\bbbld=msbm10 scaled\magstephalf

\newcommand{\bfH}{\hbox{\bbbld H}}

\newcommand{\bfR}{\hbox{\bbbld R}}

\newcommand{\hu}{\hat{u}}

\newcommand{\vN}{{\bf N}}

\newcommand{\ve}{{\bf e}}

\newcommand{\vn}{{\bf n}}

\newcommand{\vx}{{\bf x}}

\newcommand{\e}{\varepsilon}
\newcommand{\goto}{\rightarrow}
\newcommand{\ol}{\overline}

\newcommand{\be}{\begin{equation}}
\newcommand{\ee}{\end{equation}}
\newcommand{\bea}{\begin{eqnarray}}
\newcommand{\eea}{\end{eqnarray}}

\newtheorem{theorem}{Theorem}[section]
\newtheorem{lemma}[theorem]{Lemma}
\newtheorem{proposition}[theorem]{Proposition}
\newtheorem{corollary}[theorem]{Corollary}

\theoremstyle{definition}
\newtheorem{definition}[theorem]{Definition}

\theoremstyle{remark}

\numberwithin{equation}{section}

\begin{document}
\setlength{\baselineskip}{1.2\baselineskip}

\title[The asymptotic Plateau problem]
{Interior curvature estimates and the asymptotic Plateau problem in Hyperbolic space}

\author{Bo Guan}
\address{Department of Mathematics, Ohio State University,
         Columbus, OH 43210}
\email{guan@math.osu.edu}
\author{Joel Spruck}
\address{Department of Mathematics, Johns Hopkins University,
 Baltimore, MD 21218}
\email{js@math.jhu.edu}
\author{Ling Xiao}
\address{Department of Mathematics, Johns Hopkins University,
 Baltimore, MD 21218}
\email{lxiao@math.jhu.edu}
\thanks{Research  supported in part by the NSF and Simons Foundation.}

\begin{abstract}

\noindent
We show that for a  very general class of curvature functions defined in the
positive cone, the problem of finding a complete strictly locally convex
hypersurface in $\bfH^{n+1}$ satisfying $f(\kappa)=\sigma \in (0,1)$ with a
prescribed asymptotic boundary $\Gamma$ at infinity has at least one smooth
solution with uniformly bounded hyperbolic principal curvatures.
Moreover if $\Gamma$ is (Euclidean) starshaped, the solution is unique and
also (Euclidean) starshaped while if  $\Gamma$  is mean convex the solution
is unique. We also show via a strong duality theorem that analogous results
 hold in De Sitter space. A novel feature of our approach is
 a ``global interior curvature estimate''.
\end{abstract}

\maketitle

\section{Introduction}
\label{sec1}
\setcounter{equation}{0}

The asymptotic Plateau problem for complete strictly locally convex hypersurfaces
of constant Gauss curvature was initiated by Labourie \cite{Labourie} in $\bfH^3$
and by Rosenberg-Spruck \cite{RS94} in $\bfH^{n+1}$ and subsequently extended to
more general curvature functions in \cite{GSS}, \cite{GS10}, \cite{GS11},
\cite{SX12}.
In this paper we give a complete solution (Theorem \ref{th1.6}) to the asymptotic
Plateau problem for locally strictly convex hypersurfaces of constant curvature
for essentially arbitrary ``elliptic curvature functions''.
A novel feature of our work is the derivation of a ``global interior curvature
bound''  (Theorem \ref{th1.5}) that besides yielding optimal existence allows us
to infer that the convex solutions are starshaped for sharshaped asymptotic
boundary (Theorem \ref{th1.7}) and unique for mean convex asymptotic boundary
(Theorem \ref{th1.8}).
 \\

 Given $\Gamma \subset \partial_{\infty} \bfH^{n+1}$  and a smooth symmetric function
 $f$ of $n$ variables, we seek a complete  locally strictly convex hypersurface
$\Sigma$ in $\bfH^{n+1}$  satisfying
\begin{equation}
\label{eq1.10}
f(\kappa[\Sigma]) = \sigma
\ee
with the asymptotic boundary
\be \label{eq1.20}
\partial \Sigma = \Gamma
\end{equation}
where $\kappa[\Sigma] = (\kappa_1, \dots, \kappa_n)$
denotes the induced (positive) hyperbolic principal curvatures of $\Sigma$ and
$\sigma \in (0,1)$ is a constant. \\

The function $f$ is to satisfy the standard structure
conditions \cite{CNS3} in the positive cone
$K^+_n := \big\{\lambda \in \bfR^n:
   \mbox{each component $\lambda_i > 0$}\big\}$:
 \begin{equation}
\label{eq1.70}
 f > 0 \;\;\mbox{in $K^+_n$}, \;\;
f = 0 \;\;\mbox{on $\partial K^+_n$},
\end{equation}
\begin{equation}
\label{eq1.50}
f_i (\lambda) \equiv \frac{\partial f (\lambda)}{\partial \lambda_i} > 0
  \;\; \mbox{in $K^+_n$}, \;\; 1 \leq i \leq n,
\end{equation}
\begin{equation}
\label{eq1.60}
\mbox{$f$ is a concave function in $K^+_n$}.
\end{equation}
In addition, we assume that $f$ is normalized and homogeneous of degree one
\begin{equation}
\label{eq1.90}
f(1, \dots, 1) = 1,  \;\; f(t \kappa) = t f (\kappa)
\;\; \forall \, t \geq 0, \; \kappa \in K^+_n.
\end{equation}

{\em By contrast we will drop the following more technical assumption} of
\cite{GSS}, \cite{GS10}, \cite{GS11},  \cite{SX12}:
\begin{equation}
\label{eq1.100}
\lim_{R \rightarrow + \infty}
   f (\lambda_1, \cdots, \lambda_{n-1}, \lambda_n + R)
    \geq 1 + \varepsilon_0 \;\;\;
\mbox{uniformly in $B_{\delta_0} ({\bf 1})$}
\end{equation}
for some fixed $\varepsilon_0 > 0$ and $\delta_0 > 0$,
where $B_{\delta_0} ({\bf 1})$ is the ball
of radius $\delta_0$ centered at ${\bf 1} = (1, \dots, 1) \in \bfR^n$. This
technical condition is the main assumption used in the proof of boundary estimates.

We will use the upper half-space model
\[ \bfH^{n+1} = \{(x, x_{n+1}) \in \bfR^{n+1}: x_{n+1} > 0\} \]
equipped with the hyperbolic metric
\begin{equation}
\label{eq1.30}
 ds^2 = \frac{1}{x_{n+1}^2} \sum_{i=1}^{n+1}dx_i^2.
\end{equation}

Thus $\partial_\infty \bfH^{n+1}$ is naturally identified with
$\bfR^n = \bfR^n \times \{0\} \subset \bfR^{n+1}$ and (\ref{eq1.20}) may
be understood in the Euclidean sense. For convenience we say $\Sigma$ has
compact asymptotic boundary if
$\partial \Sigma \subset \partial_\infty \bfH^{n+1}$ is compact with respect
to the Euclidean metric in $\bfR^n$.\\

 In this paper all
hypersurfaces in $\bfH^{n+1}$ we consider are assumed to be
connected and orientable. If $\Sigma$ is a complete hypersurface in
$\bfH^{n+1}$ with compact asymptotic boundary at infinity, then the
normal vector field of $\Sigma$ is chosen to be the one pointing to
the unique unbounded region in $\bfR^{n+1}_+ \setminus \Sigma$, and
the (both hyperbolic and Euclidean) principal curvatures of $\Sigma$
are calculated with respect to this normal vector field.
The following relation between the hyperbolic and Euclidean principal
curvatures  is well known (see, e.g, \cite{GS10} or \cite{GS11} for a
proof)
 \be
 \label{eq1.150}
 \kappa_i = x_{n+1} \kappa^e_i + \nu^{n+1}, \;\; 1 \leq i \leq n
 \ee
at $(x, x_{n+1}) \in \Sigma$, where $\nu$ is Euclidean
unit normal vector to $\Sigma$ and $\nu^{n+1} = \nu \cdot e_{n+1}$.

One important consequence of \eqref{eq1.150} is the following result of \cite{GSS}.

\begin{theorem}
\label{th1.0}
Let $\Sigma$ be a complete locally strictly convex $C^2$ hypersurface in
$\bfH^{n+1}$ with compact asymptotic boundary at infinity.
Then $\Sigma$ is the (vertical) graph of a function
$u \in C^2 (\Omega) \cap C^0 (\ol{\Omega})$, $u > 0$ in $\Omega$
and  $u = 0$ on $\partial \Omega$, for some domain $\Omega \subset \bfR^n$.
Moreover, the function $u^2 + |x|^2$ is strictly (Euclidean) convex.
\end{theorem}

We call a hypersurface $\Sigma$ {\em locally strictly convex} if
$\kappa[\Sigma] = (\kappa_1, \dots, \kappa_n) \in K_n^+$, i.e. $\kappa_i > 0$,
$1 \leq i \leq n$, everywhere on $\Sigma$.

According to Theorem~\ref{th1.0}, it is completely general to seek solutions
of \eqref{eq1.10}, \eqref{eq1.20} among vertical graphs. 
In particular, the asymptotic boundary $\Gamma$ must be the boundary
of some bounded domain $\Omega$ in $\bfR^n$. Throughout the rest of this paper,
we assume $\Gamma = \partial \Omega \times \{0\} \subset \bfR^{n+1}$
where $\Omega$ is a bounded domain in $\bfR^n$. Unless otherwise stated, we
also assume $\partial \Omega$ is smooth. \\

In this paper we show that there is a new phenomenon of ``convexity arising
from infinity''  that forces the principal curvatures of solutions to the
asymptotic Plateau problem to be uniformly bounded. This leads to substantial
improvements of our earlier results for the convex cone $K_n^+$.
The main new technical idea is a global curvature estimate
{\em which is obtained from interior  curvature estimates}.
More precisely we have

\begin{theorem}
\label{th1.5}
Suppose $0 < \sigma < 1$  and $f$ satisfies conditions
\eqref{eq1.70}-\eqref{eq1.90}.
Let $\Sigma=\text{graph}(u)$ be a smooth locally strictly convex graph in
$\bfH^{n+1}$ satisfying \eqref{eq1.10}, \eqref{eq1.20} and
\begin{equation}
\label{eq1.85}
\nu^{n+1} \geq 2 a > 0 \; \mbox{on $\Sigma$}.
\end{equation}
For $\vx \in \Sigma$ let $\kappa_{\max} (\vx)$ be the largest principal
curvature of $\Sigma$ at $\vx$.
Then for $0<b \leq \frac{a}4$,
\begin{equation}\label{eq1.86}
\max_{\Sigma} \, \frac{u^{b} \kappa_{\max}}{\nu^{n+1}-a} \leq
  \frac{8}{a^{\frac{5}{2}}}(\sup_{\Sigma}u)^{b}.
\end{equation}
In particular, \begin{equation}\label{eq1.87}
\kappa_{\max} \leq 8 a^{- \frac{5}{2}} \;\; \mbox{on $\Sigma$}.
\end{equation}
\end{theorem}

To solve the asymptotic Plateau problem for the curvature function $f$, we apply
the existence theorem of  \cite{GS11} to the curvature function
$f^{\theta}:=\theta H_n^{\frac1{n}} +(1-\theta)f $ which satisfies conditions
\eqref{eq1.70}-\eqref{eq1.90} as well as \eqref{eq1.100}, where
$H_n(\kappa_1, \ldots, \kappa_n) = \kappa_1 \cdots \kappa_n$ corresponds
to the Gauss curvature.
We obtain a complete strictly locally convex solution
$\Sigma^{\theta}=\text{graph}(u^{\theta})$ in $\bfH^{n+1}$ satisfying
(\ref{eq1.10})-(\ref{eq1.20}) with $f$ replaced by $f^{\theta}$ with bounded
principal curvatures depending on $\theta$. Moreover,
$u^{\theta}\in C^{0,1}(\ol{\Omega})$,
$(u^{\theta})^2 \in C^{\infty} (\Omega)\cap C^{1,1}(\ol{\Omega})$
and $u^{\theta} + |Du^{\theta}| \leq C$ independent of $\theta$.
Using Theorem \ref{th1.5}, we find that $u^{\theta}|D^2 u^{\theta}| \leq C$ where
$C$ is independent of $\theta$. We can now let $\theta$ tend to $0$ to obtain
the following existence theorem for $\Gamma=\partial \Omega$ satisfying a
uniform exterior ball condition.

\begin{theorem}
\label{th1.6}
Suppose $0 < \sigma < 1$, $\Omega$ satisfies a uniform exterior ball condition
 and that $f$ satisfies conditions (\ref{eq1.70})-(\ref{eq1.90})
in $K_n^+$. There exists a complete locally strictly convex hypersurface
$\Sigma=\text{graph}(u)$ in $\bfH^{n+1}$ satisfying
(\ref{eq1.10})-(\ref{eq1.20})
with uniformly bounded principal curvatures
\begin{equation}
\label{eq1.190}
 C^{-1} \leq \kappa_i \leq C \;\; \mbox{on $\Sigma$}.
\end{equation}
Moreover, $u \in C^{\infty} (\Omega)\cap C^{0,1}(\ol{\Omega})$,
$u^2 \in  C^{1,1}(\ol{\Omega})$,
$u|D^2 u| +|Du| \leq C$
and, if $\partial \Omega \in C^2$
\begin{equation}
\label{eq1.200}
\sqrt{1 + |Du|^2} = \frac{1}{\sigma}
\hspace{.1in} \mbox{on $\partial \Omega$}.
\end{equation}
\end{theorem}

Note that no uniqueness of solutions is asserted. In \cite{GS11} we showed
uniqueness if
\[\sum f_i > \sum \lambda_i^2 f_i \,\,\mbox {in $K_n^+ \cap \{ 0<f<1\}$}.\]
In particular, uniqueness holds for the curvature quotients
$f=(\frac{H_n}{H_k})^{\frac1{n-l}}$ with $k=n-1$ or $k=n-2$.
Here $H_k$ is the normalized $k$-th elementary symmetric function.
We can prove the following general uniqueness when $\Omega$ is  mean convex.

\begin{theorem}
\label{th1.8}
Assume $\Omega$ is a $C^{2,\alpha}$ mean convex domain, that is, the Euclidean mean
curvature $\mathcal{H}_{\partial \Omega} \geq 0$. Then the solution $\Sigma$ of
Theorem \ref{th1.6} is unique.
\end{theorem}

There is uniqueness if $\partial\Omega$ is strictly (Euclidean) starshaped about the
origin. This is a well-known fact. However we can say much more in this case.

\begin{theorem}
\label{th1.7}
Let $\partial\Omega\in C^1$ be strictly (Euclidean) starshaped about the origin.
Then the unique solution given in Theorem \ref{th1.6} is strictly (Euclidean)
starshaped about the origin, i.e. $\vx \cdot \nu>0$.
\end{theorem}

We end with an application of Theorem \ref{th1.6} to the existence of constant
curvature spacelike hypersurfaces in de Sitter space.
There  is a natural asymptotic Plateau problem dual to
\eqref{eq1.10}-\eqref{eq1.20} for strictly spacelike hypersurfaces
 \cite{SX12} which takes place in the steady state  subspace
$\mathcal{H}^{n+1}  \subset dS_{n+1}$ of de Sitter space.
Following Montiel \cite{Montiel}, there is a  halfspace model which identifies
$\mathcal{H}^{n+1} $ with $ \bfR^{n+1}_+$ endowed with the Lorentz metric
\be \label{eq1.220}
ds^2=\frac{1}{y_{n+1}^2}(dy^2-dy_{n+1}^2).
\ee
It is important to note that {\em the isometry from $\mathcal{H}^{n+1} $ to the
halfspace model reverses the time orientation}. The dual asymptotic Plateau
problem seeks to find a strictly spacelike hypersurface $S$ satisfying
\begin{equation}
\label{eq1.240}
f(\kappa[S]) = \sigma > 1, \;\;
\partial S = \Gamma
\end{equation}
where $\kappa[S]$ denotes the principal curvatures of $S$ in the induced
de Sitter metric.\\

If $S$ is a complete spacelike hypersurface in
$\mathcal{H}^{n+1}$ with compact asymptotic boundary at infinity, then the
normal vector field N of $S$ is chosen to be the one pointing to
the unique unbounded region in $\bfR^{n+1}_+ \setminus S$, and
the de Sitter principal curvatures of $S$
are calculated with respect to this normal vector field.\\

Because $S$ is strictly spacelike,  we are essentially forced to take
 $\Gamma=\partial V$ where
$V \subset \bfR^n$ is a  bounded domain and seek $S$ as  the
graph of  a ``spacelike'' function $v$ 
\be \label{eq1.260}
S=\{(y,y_{n+1}): ~ y_{n+1}=v(y), ~ y \in V\},
\hspace{.1in} |\nabla v| <1 \; \mbox{in $\ol{V}$}.
 \ee

In \cite{SX12} we have computed the  first and second fundamental forms of S with
respect to the induced de Sitter metric.  We use
 \[ X_i=e_i+v_i e_{n+1},\;\;\  N=v \nu= v\frac{v_ie_i+e_{n+1}}{w},\]
where $w=\sqrt{1-|\nabla v|^2}$ and $\nu$ is the normal vector field of S viewed
as a Minkowski space $R^{n,1}$ graph.
The first and second fundamental forms $g_{ij}$  and $h_{ij}$ are given by
 \be \label{eq1.280}
 g_{ij}=\langle X_i, X_j\rangle_D=\frac{1}{v^2}(\delta_{ij}-v_iv_j),
\ee
\be\label{eq1.300}
\begin{aligned}
h_{ij}&=\langle\nabla_{X_i}X_j, v \nu\rangle_D
=\frac{1}{v^2w}(\delta_{ij}-v_iv_j-vv_{ij})
\end{aligned}
\ee
respectively. Note that from \eqref{eq1.300}, $S$ is locally strictly convex if
and only if
\be \label{eq1.320}
\mbox{ $|y|^2-v^2 $ is a (Euclidean) locally strictly convex function.}
\ee

There is a well known Gauss map duality for locally strictly convex hypersurfaces
in $dS_{n+1}$.
For our purposes we will need a very concrete formulation of this duality
\cite{SX12}.  Montiel \cite{Montiel} showed that if we use the upper halfspace
representation for both $\mathcal{H}^{n+1}$ and $\bfH^{n+1}$, the Gauss map $N$
corresponds to the map $L : S \goto \bfH^{n+1}$ defined by
 \be \label{eq1.320'}
  L((y,v(y)))=(y-v(y)\nabla v(y), v(y)\sqrt{1-|\nabla v|^2}), \;\; y \in V .
  \ee
 We now identify the map $L$ in terms of a hodograph map and its associated
Legendre transform.
 Let $p(y)= \frac{1}{2}(|y|^2-v(y)^2)$; since
 $p$ is strictly convex in the Euclidean sense by \eqref{eq1.320},
 its gradient map $\nabla p: V \subset \bfR^n \goto \bfR^n$ is globally one to one.
 Define
 \be \label {eq1.340}
 x =\nabla p(y), \;\; u(x) := v(y)\sqrt{1-|\nabla v(y)|^2}, \;\; y \in V.
\ee
Then $u$ is well defined in $\Omega := \nabla p (V)$.
The associated Legendre transform is the function $q(x)$ defined in  $\Omega$
  by $ p(y)+q(x)= x \cdot y$ or $q(x)=-p(y)+y\cdot \nabla p(y)$.

\begin{theorem}\cite{SX12}.
\label{th1.10}
Let L be defined by \eqref{eq1.320'} x by \eqref{eq1.340}. Then the image of S
under L
is the hyperbolic locally strictly convex graph in $\bfH^{n+1}$
\[ \Sigma=\{(x, u(x))\in \bfR^{n+1}_+: u \in C^{\infty}(\ol{\Omega}), \; u(x)>0\} \]
 with principal curvatures
$\kappa_i^*=\kappa_i^{-1}$. Here $\kappa_1, \ldots \kappa_n$ are the principal
curvatures of $S$ with respect to the induced de Sitter metric.
Moreover the inverse map $L^{-1}: \Sigma \goto S$ 
\[L^{-1}((x,u(x)))=(x+u(x) D u(x), u(x)\sqrt{1+|D u(x)|^2}),
\hspace{.1in} x \in \Omega\]
is the dual Legendre transform and hodograph map $y = D q(x)$,
$q(x)=\frac12(|x|^2+u(x)^2)$.
\end{theorem}

Note that when $\Sigma=\text{graph}(u)$ over $\Omega$ is a strictly locally
convex solution of the asymptotic
Plateau problem~(\ref{eq1.10})-(\ref{eq1.20}) in $\bfH^{n+1}$,
then its Gauss image $S=\text{graph}(v)$ is a locally strictly convex
spacelike graph also defined over $\Omega$ which solves the asymptotic Plateau
problem $f^*(\kappa)=\frac1{\sigma}>1$.  We now define $f^*$.

\begin{definition}
Given a curvature function $f(\kappa)$ in the positive cone $K_n^+$, define the
dual curvature function $f^*(\kappa)$ by
\be \label {eq1.360}
f^*(\kappa):=\frac{1}{f(\kappa_1^{-1},\ldots, \kappa_n^{-1})},
\hspace{.1in} \kappa \in K_n^+.
\ee
\end{definition}

Note that $f^*$ may in fact be naturally defined in a cone $K\supseteq K_n^+$.
For example if $f(\kappa)=\big(\frac{H_n}{H_l}\big)^{\frac1{n-l}},\,n>l \geq 0$
defined in $K_n^+$, then
\[ f^*(\kappa)= \big(H_{n-l}\big)^{\frac1{n-l}}\]
is in fact  defined in the standard Garding cone $K=\Gamma_{n-l}$.
\\

Using the duality Theorem \ref{th1.10} we can transplant Theorem \ref{th1.6}
to $\mathcal{H}^{n+1}$.

 \begin{theorem}
\label{th4.0}
Let $\sigma>1$. Suppose that $\Omega$ satisfies a uniform exterior ball
condition
 and that $f$ satisfies conditions (\ref{eq1.70})-(\ref{eq1.90}) in $K_n^+$.
There exists a complete locally strictly convex  spacelike graph
$S=\text{graph}(v)$ in $\mathcal{H}^{n+1}$ satisfying $f^*(\kappa) = \sigma$
and $\partial S = \Gamma$ with uniformly bounded principal curvatures
$C^{-1} \leq \kappa_i \leq C$ on $S$.
Furthermore, $v\in C^{\infty} (\Omega) \cap C^{0,1}(\ol{\Omega})$,
$v^2 \in C^{1,1}(\ol{\Omega})$, $v|D^2v| +|Dv| \leq C$
and, if $\partial \Omega \in C^2$
\begin{equation}
\label{eq1.200'}
\sqrt{1 - |Dv|^2} = \frac{1}{\sigma}
\hspace{.1in} \mbox{on $\partial \Omega$}.
\end{equation}
\end{theorem}

\begin{corollary}
\label{cor1.2}
Suppose that $\Omega$ satisfies a uniform exterior ball condition.
There exists a complete locally strictly convex spacelike hypersurface
$S$ in $\mathcal{H}^{n+1}$ satisfying
\[ (H_l)^{\frac1l}=\sigma>1, \hspace{.1in} 1 \leq l \leq  n\]
with $\partial S=\Gamma$ and having uniformly bounded principal curvatures
$C^{-1} \leq \kappa_i \leq C$ on $S$.
Moreover, $S = \text{graph}(v)$ with
$v \in C^\infty (\Omega) \cap C^{0,1} (\bar{\Omega})$,
$v^2 \in C^{1,1}(\ol{\Omega})$, $v|D^2v| +|Dv| \leq C$.
Further, if $l=1$ or $l=2$ (corresponding to mean curvature and normalized scalar
curvature) or if $\partial \Omega$ is mean convex, we have uniqueness among convex
solutions and even among all solutions (convex or not) if
$\Omega$ is simply-connected.
\end{corollary}

The uniqueness part of Corollary~\ref{cor1.2} follows from Theorem 1.6 of
\cite{GS11} or Theorem~\ref{th1.8} and a continuous deformation argument as
used in \cite{RS94}. Montiel \cite{Montiel} proved existence for $H=\sigma>1$
(mean curvature)
assuming $\partial \Omega$ is mean convex. Our result shows that for arbitrary
$\Omega$ there is always a unique locally strictly convex solution. If $\Omega$
is mean convex the solutions constructed by Montiel must agree with the ones
we construct.\\

An outline of the paper is as follows. In Section~\ref{sec2} we recall
some important identities and estimates, most of them from \cite{GS11},
needed in the proof of our main technical result (Theorem \ref{th4.10}),
the ``global interior curvature estimate''.
These identities and formulas are interesting and important in themselves and
will orient the reader to our point of view.
The proof of Theorem \ref{th4.10} is carried our in Section~\ref{sec5};
Theorem~\ref{th1.5} follows immediately.
 Theorem \ref{th1.7} and Theorem \ref{th1.8} are proved in Sections~\ref{sec6}
and \ref{sec7}, respectively;
 the use of Theorem \ref{th1.5} is essential in these proofs.

In the rest sections $f$ is always assumed to satisfy
(\ref{eq1.70})-(\ref{eq1.90}) in $K_n^+$.


\bigskip

\section{Formulas on hypersurfaces and some basic identities}
\label{sec2}
\setcounter{equation}{0}

In this section we recall some basic properties of solutions of \eqref{eq1.10}
derived in \cite{GS11} that will be needed in the following sections
to prove our main results.\\

Let $\Sigma$ be a hypersurface in $\bfH^{n+1}$. We shall use $g$ and $\nabla$
to denote the induced hyperbolic metric and Levi-Civita connection
on $\Sigma$, respectively.

Let $\vx$ and $\nu$ be the position vector and Euclidean unit normal vector
 of $\Sigma$ in $\bfR^{n+1}$, respectively and set
\[ u = \vx \cdot \ve, \;\; \nu^{n+1} = \ve \cdot \nu \]
where $\ve$ is the unit vector in the positive
$x_{n+1}$ direction in $\bfR^{n+1}$, and `$\cdot$' denotes the Euclidean
inner product in $\bfR^{n+1}$.
We refer $u$ as the {\em height function} of $\Sigma$.
The hyperbolic unit normal vector is $\vn = u \nu$.

Let $\tau_1, \ldots, \tau_n$ be local frames.
The metric and second fundamental form of $\Sigma$ are respectively
given by
\begin{equation}
\label{eq2.10} g_{ij} = \langle \tau_i, \tau_j \rangle, \;\;
   h_{ij} = \langle D_{\tau_i} \tau_j, {\bf n} \rangle
              = - \langle D_{\tau_i} {\bf n}, \tau_j \rangle
\end{equation}
where $D$ denotes the Levi-Civita connection of $\bfH^{n+1}$.
Throughout the paper we assume $\tau_1, \ldots, \tau_n$ are orthonormal
so $g_{ij} = \delta_{ij}$.
The principal curvatures of $\Sigma$ are the eigenvalues of
the second fundamental form $\{h_{ij}\}$ with respect to the metric
$\{g_{ij}\}$. The following formula is derived in \cite{GS11}
\begin{equation}
\label{eq2.100}
\begin{aligned}
\nabla_{ij} \frac{1}{u}
 \,&  = \frac{1}{u} (g_{ij} - \nu^{n+1} h_{ij}).
 \end{aligned}
\end{equation}

Let $\mathcal{S}$ be the space of $n \times n$ symmetric matrices
and $\mathcal{S}^+ = \{A \in \mathcal{S}: \lambda (A) \in K_n^+\}$,
where $\lambda (A) = (\lambda_1, \dots, \lambda_n)$ are the
eigenvalues of $A$. Let $F$ be the function defined by
\begin{equation}
\label{eq2.110}
F (A) = f (\lambda (A)), \;\; A \in \mathcal{S}^+
\end{equation}
and denote
\begin{equation}
\label{eq2.120}
F^{ij} (A) = \frac{\partial F}{\partial a_{ij}} (A), \;\;
  F^{ij, kl} (A) = \frac{\partial^2 F}{\partial a_{ij} \partial a_{kl}} (A).
\end{equation}
We have $F^{ij} (A) = f_i (\lambda(A)) \delta_{ij}$
when $A$ is diagonal. Moreover,
\begin{equation}
\label{eq2.130}
 F^{ij} (A) a_{ij} = \sum f_i (\lambda (A)) \lambda_i = F (A),
\end{equation}
\begin{equation}
\label{eq2.140}
F^{ij} (A) a_{ik} a_{jk} = \sum f_i (\lambda (A)) \lambda_i^2.
\end{equation}

Equation~\eqref{eq1.10} can therefore be rewritten locally in the form
\begin{equation}
\label{eq2.150}
F (h_{ij}) = \sigma.
\end{equation}
Denote $F^{ij} = F^{ij} (h_{ij})$, $F^{ij, kl} = F^{ij, kl} (h_{ij})$.

\begin{lemma}[\cite{GS11}]
\label{lem2.10}
Let $\Sigma$ be a smooth hypersurface in $\bfH^{n+1}$ satisfying
\eqref{eq1.10}. Then
\begin{equation}
\label{eq2.160}
  F^{ij} \nabla_{ij} \frac{1}{u}
    = - \frac{\sigma \nu^{n+1}}{u} + \frac{1}{u} \sum f_i,
\end{equation}
\begin{equation}
\label{eq2.170}
  F^{ij} \nabla_{ij} \frac{\nu^{n+1}}{u} =
\frac{\sigma}{u} - \frac{\nu^{n+1}}{u} \sum f_i \kappa_i^2.
\end{equation}
\end{lemma}

Using Lemma \ref{lem2.10} one derives the following important maximum
principle.

\begin{theorem}[\cite{GS11}]
\label{th2.1}
Let $\Sigma$ be a smooth strictly locally convex hypersurface in
$\bfH^{n+1}$
satisfying equation~\eqref{eq1.10}. Suppose $\Sigma$ is globally a graph:
$\Sigma = \{(x, u (x)): x \in \Omega\}$
where $\Omega$ is a domain in $\bfR^n \equiv \partial \bfH^{n+1}$. Then
\be
\label{eq2.180}
F^{ij}\nabla_{ij}\frac{\sigma - \nu^{n+1}}{u}
   \geq \sigma(1-\sigma)\frac{(\sum f_i -1)}u  \geq 0.
\ee
\end{theorem}

Upper and lower bounds on $\partial \Omega$ for
$\eta:=\frac{\sigma-\nu^{n+1}}u$ follow from the following lemma which is
based on comparisons with equidistant sphere solutions.

\begin{lemma}
\label{lem2.20}
 Assume that $\partial \Sigma $ satisfies a uniform interior and/or
exterior ball condition and let $u$ denote the height function of
$\Sigma$ with $u=\e$ on $\partial \Omega$.
Then for $\e \geq 0$ sufficiently small,
\begin{equation}
\label{eq2.200}
 - \frac{\e \sqrt{1-\sigma^2}}{r_2}
   - \frac{\e^2 (1+\sigma)}{r_2^2}
   < \nu^{n+1}-\sigma
        <  \frac{\e \sqrt{1-\sigma^2}}{r_1} +
        \frac{\e^2 (1-\sigma)}{r_1^2}
\;\;\; \mbox{on $\partial \Sigma$}
\end{equation}
where $r_2$ and $r_1$ are
the maximal radii of exterior and interior spheres to
$\partial \Omega$, respectively.
In particular,  $\nu^{n+1} \rightarrow \sigma$
on $\partial \Sigma$ as $\e \rightarrow 0$.
\end{lemma}

\begin{corollary}\label{cor2.1}
\begin{equation}
\label{eq32.220}
\eta:=\frac{\sigma - \nu^{n+1}}{u}
  \leq \sup_{\partial \Sigma} \frac{\sigma - \nu^{n+1}}{u}
\;\; \mbox{on $\Sigma$}.
\end{equation}
Moreover, if $u = \epsilon>0$ on  $\partial \Omega$ (satisfying
 a uniform exterior ball condition), then there exists
$\epsilon_0 > 0$ depending only on $\partial \Omega$, such that for all
$\epsilon \leq \epsilon_0$,
\begin{equation}
\label{eq2.240}
\frac{\sigma - \nu^{n+1}}{u}
   \leq \frac{\sqrt{1-\sigma^2}}{r_2}+\frac{\e(1+\sigma)}{r_2^2}
\;\; \mbox{on $\Sigma$}
\end{equation}
where $r_2$ is the maximal radius of exterior tangent spheres to
$\partial \Omega$.
\end{corollary}

\begin{proposition}
\label{prop2.2}
Let $\Sigma$ be a smooth strictly locally convex graph
\[ \Sigma = \{(x, u (x)): x \in \Omega\} \]
in $\bfH^{n+1}$ satisfying $u\geq \e \,\, \mbox{in $\Omega$},\,
  u=\e \,\mbox{on $\partial \Omega$}$.
 Then at an interior maximum of $\frac{u}{\nu^{n+1}}$ we have
$\frac{u}{\nu^{n+1}} \leq \max_{\Omega} u$.
Hence for $\e$ small compared to $\sigma$,
\be
\label{eq2.260}
  \nu^{n+1}\geq \frac{u}{ \max_{\Omega} u}    \hspace{.1in} \mbox{in $\Omega$}
\ee
\end{proposition}

\begin{proof}
Let $h=\frac{u}{\nu^{n+1}}=uw$ and suppose that $h$ assumes its maximum at
an interior point $x_0$. Then at $x_0$,
\[ \partial_i h = u_i w+u \frac{u_k u_{ki}}w
                = (\delta_{ki}+u_k u_i+u u_{ki})\frac{u_k}w = 0
\;\; \forall \; 1 \leq i \leq n.\]
Since $\Sigma$ is strictly locally convex, this implies that $\nabla u=0$
at $x_0$ so the proposition follows immediately from Corollary \ref{cor2.1}.
\end{proof}

Combining Theorem~\ref{th2.1} and Proposition~\ref{prop2.2} gives

\begin{corollary}
\label{cor2.3}
Let $\Sigma$ be a smooth strictly locally convex graph
\[ \Sigma = \{(x, u (x)): x \in \Omega\} \]
in $\bfH^{n+1}$ satisfying $u\geq \e \,\, \mbox{in $\Omega$},\,
  u=\e \,\mbox{on $\partial \Omega$}$.
 Assume that $\partial \Omega$ satisfies a uniform exterior ball condition.
 Then for $\e$ sufficiently small compared to $\sigma$
\be
\label{eq2.280}
  \nu^{n+1}\geq  2a:=\frac{\sigma } {1+M \max_{\Omega} u}
\ee
where $M= \frac{\sqrt{1-\sigma^2}}{r_2}+\frac{\e(1+\sigma)}{r_2^2}$.
\end{corollary}
\begin{proof} By Theorem~\ref{th2.1} we have $\nu^{n+1} \geq \sigma-Mu $ while by
Proposition~\ref{prop2.2} we have  $\nu^{n+1} \geq \frac{u}{ \max_{\Omega} u} $.
Hence if $ u\leq \lambda \sigma$ we find $\nu^{n+1} \geq \sigma(1-\lambda M)$
while if $u\geq \lambda \sigma$ we find
$\nu^{n+1} \geq \frac{\lambda \sigma}{ \max_{\Omega} u} $.
Choosing $\lambda = \frac{ \max_{\Omega} u}{1+M  \max_{\Omega} u}$
completes the proof.
\end{proof}

\bigskip

\section{The global interior curvature estimate }
\label{sec5}
\setcounter{equation}{0}

In this section we prove an interior curvature estimate (see Theorem \ref{th4.10}
below) for the largest principal curvature of locally strictly convex graphs
satisfying $f(\kappa)=\sigma$.  What is remarkable is that the bound we obtain is
independent of the ``cutoff " function $u^b$ which vanishes at $\partial \Omega$.
Hence we can let $b$ tend to zero to prove the global estimate Theorem \ref{th1.5}.
\\

Let $\Sigma$ be a smooth strictly locally convex  hypersurface in $\bfH^{n+1}$
satisfying
$f(\kappa)=\sigma$ with $\partial \Sigma \subset \partial_{\infty}\bfH^{n+1}$.
For a fixed point $\vx_0 \in \Sigma$ we choose a local orthonormal frame
$\tau_1, \ldots, \tau_n$
around $\vx_0$ such that $h_{ij} (\vx_0) = \kappa_i \delta_{ij}$.
The calculations below are done at $\vx_0$. For convenience we shall
write $v_{ij} = \nabla_{ij} v$,  $h_{ijk} = \nabla_k h_{ij}$,
$h_{ijkl} = \nabla_{lk} h_{ij} = \nabla_l \nabla_k h_{ij}$, etc.\\

Since $\bfH^{n+1}$ has constant sectional curvature $-1$, by the Codazzi
and Gauss equations we have $ h_{ijk} = h_{ikj}$ and
\begin{equation}
\label{eq4.10}
\begin{aligned}
 h_{iijj} = \,& h_{jjii} + (h_{ii} h_{jj} - 1) (h_{ii} - h_{jj})
          =  h_{jjii} + (\kappa_i \kappa_j - 1) (\kappa_i - \kappa_j).
 \end{aligned}
 \end{equation}
Consequently for each fixed $j$,
\begin{equation}
\label{eq4.25}
 F^{ii} h_{jjii} =  F^{ii} h_{iijj} + (1 + \kappa_j^2) \sum f_i \kappa_i
     - \kappa_j \sum f_i -  \kappa_j \sum \kappa_i^2  f_i.
\end{equation}

\begin{theorem}
\label{th4.10}
Let $\Sigma$ be a smooth strictly locally convex graph in $\bfH^{n+1}$
satisfying $f(\kappa)=\sigma$,
$\partial_{\infty} \Sigma \subset \partial_{\infty}\bfH^{n+1}$ and
\begin{equation}
\label{eq4.30}
\nu^{n+1} \geq 2 a > 0 \; \mbox{on $\Sigma$}.
\end{equation}
For $\vx \in \Sigma$ let $\kappa_{\max} (\vx)$ be the largest principal
curvature of $\Sigma$ at $\vx$. Then for $0<b \leq \frac{a}4$,
\begin{equation}
\max_{\Sigma} \, u^{b}\frac{\kappa_{\max}}{\nu^{n+1}-a} \leq
  \frac{8}{a^{\frac52}}(\sup_{\Sigma}u)^{b}.
\end{equation}
\end{theorem}

\begin{proof}
Let
\begin{equation}
\label{eq4.35}
 M_0 = \max_{\vx \in \Sigma}u^b \frac{\kappa_{\max} (x) }{\nu^{n+1} - a}.
\end{equation}
Then $M_0 > 0$ is attained at an interior point $\vx_0 \in \Sigma$.
Let $\tau_1, \ldots, \tau_n$ be a local orthonormal frame
around $\vx_0$ such that $h_{ij} (\vx_0) = \kappa_i \delta_{ij}$,
where $\kappa_1, \ldots, \kappa_n$ are the principal curvatures of $\Sigma$
at $\vx_0$. We may assume
$\kappa_1 = \kappa_{\max} (\vx_0)$.
Thus, at $\vx_0,\, u^b \frac{h_{11}}{\nu^{n+1}-a}$ has a local maximum and so
\begin{equation}
\label{eq4.40}
\frac{h_{11i}}{h_{11}} +b\frac{u_i}u - \frac{\nabla_i \nu^{n+1}}{\nu^{n+1}-a} = 0,
\end{equation}
\begin{equation}
\label{eq4.45}
\frac{h_{11ii}}{h_{11}}+b\frac{u_{ii}}u-\frac{\nabla_{ii} \nu^{n+1}}{\nu^{n+1}-a}
-(b+b^2)\frac{u_i^2}{u^2}+2b \frac{u_i}u \frac{\nabla_i \nu^{n+1}}{\nu^{n+1}-a}
\leq 0.
\end{equation}

Using \eqref{eq4.25}, we find after differentiating the equation
$F(h_{ij})=\sigma$ twice that
at $\vx_0$,
\begin{equation}
\label{eq4.50}
F^{ii}h_{11ii}= - F^{ij,rs}h_{ij1} h_{rs1} + \sigma (1 + \kappa_1^2)
                - \kappa_1 \Big(\sum f_i + \sum \kappa_i^2 f_i\Big).
\end{equation}
By Lemma \ref{lem2.10} we immediately derive
\begin{equation}
\label{eq4.60}
\begin{aligned}
  F^{ij} \nabla_{ij} \nu^{n+1}
   = \,& \frac2{u} F^{ij}\nabla_i u \nabla_j \nu^{n+1}+\sigma(1+ (\nu^{n+1})^2)\\
     \,& - \nu^{n+1}\Big(\sum f_i+\sum f_i \kappa_i^2\Big),
  \end{aligned}
\end{equation}
     \begin{equation}
\label{eq4.60'}
\begin{aligned}
 F^{ij}\frac{ \nabla_{ij}u}u
   = \,& 2\sum f_i\frac{u_i^2}{u^2}+\sigma\nu^{n+1}-\sum f_i.
 \end{aligned}
\end{equation}

By \eqref{eq4.45}-\eqref{eq4.60'} we find
\begin{equation}
\label{eq4.70}
\begin{aligned}
  0 \geq \,& -F^{ij,rs}h_{ij1}h_{rs1}+\sigma \Big(1+\kappa_1^2
             -\frac{1+(\nu^{n+1})^2}{\nu^{n+1}-a}\kappa_1\Big)\\
        & \, + \frac{a\kappa_1}{\nu^{n+1}-a} \Big(\sum f_i
             +\sum \kappa_i^2 f_i\Big)-b\kappa_1\sum f_i\\
        & \, +(b-b^2)\kappa_1 \sum f_i \frac{u_i^2}{u^2}
  -\frac{(2-2b)\kappa_1}{\nu^{n+1}-a} F^{ij}\frac{u_i}{u} \nabla_j \nu^{n+1}.
   \end{aligned}
 \end{equation}
 Next we use an inequality due
to Andrews~\cite{Andrews94} and Gerhardt~\cite{Gerhardt96} which states
\begin{equation}
\label{eq4.80}
 - F^{ij,kl} h_{ij1} h_{kl,1}
 \geq  \sum_{i \neq j} \frac{f_i - f_j}{\kappa_j - \kappa_i} h_{ij1}^2
 \geq 2 \sum_{i\geq 2} \frac{f_i - f_1}{\kappa_1 - \kappa_i} h_{i11}^2.
\end{equation}

Recall that (see \cite{GS11})
  \[\nabla_i \nu^{n+1}=\frac{u_i}u (\nu^{n+1}-\kappa_i).\]
Thus at $\vx_0$ we obtain from \eqref{eq4.40}
\begin{equation} \label{eq4.90}
  h_{11i}  = \kappa_1\frac{u_i}u \Big(\frac{\nu^{n+1}-\kappa_i}{\nu^{n+1}-a}-b\Big).
  \end{equation}
 Inserting this into \eqref{eq4.80} we derive
\begin{equation}
\label{eq4.100}
 - F^{ij,kl} h_{ij1} h_{kl,1}
  \geq 2{\kappa_1}^2
         \sum_{i\geq 2} \frac{f_i-f_1}{\kappa_1-\kappa_i}
         \,\frac{u_i^2}{u^2} \Big(\frac{\kappa_i-\nu^{n+1}}{\nu^{n+1}-a}+b\Big)^2.
 \end{equation}
Note that we may write
\begin{equation} \label{eq4.110}
\sum f_i +\sum \kappa_i^2 f_i=(1-(\nu^{n+1})^2)\sum f_i
+\sum (\kappa_i-\nu^{n+1})^2 f_i +2\sigma \nu^{n+1}.
\end{equation}
Combining \eqref{eq4.80}, \eqref{eq4.100} and \eqref{eq4.110} gives at $\vx_0$
\begin{equation}
\label{eq4.120}
\begin{aligned}
0 \geq \,& \sigma \Big(1 + \kappa_1^2
           -\frac{1+(\nu^{n+1})^2}{\nu^{n+1}-a} \kappa_1\Big) -b\kappa_1 \sum f_i\\
       \,& + (b-b^2)\sum f_i \frac{u_i^2}{u^2} + \frac{a\kappa_1}{2(\nu^{n+1}-a)}
             \Big(\sum f_i +\sum \kappa_i^2 f_i\Big)\\
       \,& + \frac{a\kappa_1}{2(\nu^{n+1}-a)} \Big((1-(\nu^{n+1})^2) \sum f_i
           + \sum (\kappa_i-\nu^{n+1})^2 f_i +2\sigma \nu^{n+1}\Big)\\
       \,& + 2{\kappa_1}^2 \sum_{i\geq 2} \frac{f_i-f_1}{\kappa_1-\kappa_i}
          \,\frac{u_i^2}{u^2} \Big(\frac{\kappa_i-\nu^{n+1}}{\nu^{n+1}-a}+b\Big)^2 \\
       \,& + (2-2b)\kappa_1 \sum f_i \frac{u_i^2}{u^2}
            \frac{\kappa_i-\nu^{n+1}}{\nu^{n+1}-a}.
     \end{aligned}
 \end{equation}

Note that (assuming $\kappa_1 \geq \frac2{a}$ and $b\leq \frac{a}4$) all the terms
of \eqref{eq4.120} are positive except possibly the ones in the last sum involving
$(\kappa_i-\nu^{n+1})$ and only if $\kappa_i <\nu^{n+1}$.

For $\theta \in (0, 1)$ to be chosen later, define
\[  \begin{aligned}
J & = \{i:  \kappa_i - \nu^{n+1} < 0,
            \; f_i < \theta^{-1} f_1\}, \\
L & = \{i:  \kappa_i - \nu^{n+1} < 0,
            \; f_i \geq \theta^{-1} f_1\}.
  \end{aligned} \]
Since $\sum u_i^2/u^2 = |\tilde{\nabla} u|^2 = 1 - (\nu^{n+1})^2 \leq 1$,
$\nu^{n+1} \geq 2a$ and $\kappa_i  f_i \leq \sigma$ for each $i$, we derive
\begin{equation}
\label{eq4.140}
 \sum_{i \in J} (\kappa_i - \nu^{n+1}) f_i \frac{u_i^2}{u^2}
  \geq -\frac{f_1}{\theta} \geq -\frac{\sigma}{\theta \kappa_1},
\end{equation}
and
\begin{equation}
\label{eq4.150}
\begin{aligned}
2 \kappa_1^2
      \,& \sum_{i\in L} \frac{f_i - f_1}{\kappa_1 - \kappa_i}
          \frac{u_i^2}{u^2} \Big(\frac{\kappa_i-\nu^{n+1}}{\nu^{n+1}-a}+b\Big)^2
          + (2-2b)\kappa_1 \sum_{i\in L} f_i \frac{u_i^2}{u^2}
           \frac{\kappa_i-\nu^{n+1}}{\nu^{n+1}-a}\\
 \geq \,& 2 (1 - \theta) \kappa_1
         \sum_{i \in L} f_i \frac{u_i^2}{u^2}
              \Big(\frac{\kappa_i - \nu^{n+1}}{\nu^{n+1}-a}\Big)^2
          + (2 + 2b -4 b \theta)\kappa_1 \sum_{i\in L} f_i \frac{u_i^2}{u^2}
           \frac{(\kappa_i-\nu^{n+1})}{\nu^{n+1}-a}\\
  \geq \,& \frac{2 \kappa_1}{(\nu^{n+1}-a)^2} \sum_{i\in L} f_i
           \frac{u_i^2}{u^2}(\kappa_i^2-(a+\nu^{n+1})\kappa_i+a\nu^{n+1})\\
        \,&  - \frac{2\theta}a \frac{\kappa_1}{\nu^{n+1}-a}
               \sum_{i \in L}f_i (\kappa_i - \nu^{n+1})^2
             + 2b(1-2\theta) \kappa_1 \sum_{i\in L} f_i \frac{u_i^2}{u^2}
               \frac{(\kappa_i-\nu^{n+1})}{\nu^{n+1}-a}\\
  \geq \,& - \frac{6\sigma}a \kappa_1
             - \frac{2b\kappa_1 (1-(\nu^{n+1})^2)}{\nu^{n+1}-a} \sum f_i
             - \frac{2\theta \kappa_1}{a (\nu^{n+1}-a)}
               \sum_{i\in L} f_i (\kappa_i- \nu^{n+1})^2.
\end{aligned}
\end{equation}\\
We now fix $\theta=\frac{a^2}4$ and $0 < b \leq \frac{a}4$.
From (\ref{eq4.140}) and (\ref{eq4.150}) we see that the right hand side of
\eqref{eq4.120} at $\vx_0$ is strictly greater than
\be \label{eq4.170}
\sigma \Big(1+\kappa_1^2 -\frac{8}a \kappa_1 -\frac{8}{a^3}\Big).
\ee
Then \eqref{eq4.170} is strictly positive if for example
$\kappa_1\geq 8 a^{- \frac{3}{2}}$.
Therefore $\kappa_1 \leq 8 a^{- \frac{3}{2}}$ at $\vx_0$,
completing the proof of Theorem \ref{th4.10}.
\end{proof}

\bigskip

\section{Strict Euclidean starshapedness for convex solutions}
\label{sec6}
\setcounter{equation}{0}

In this section we prove Theorem \ref{th1.7} by direct construction in
Theorem~\ref{th6.1} below of a strictly starshaped locally strictly convex
solution with boundary in the horosphere $\{x_{n+1}=\e\}$.
By compactness and uniqueness we can then pass to the limit as $\e$ tends to zero.
We use the continuity method by deforming from the horosphere solution
$u\equiv \e$ for $\sigma=1$. Under this deformation we will show that the property
of being strictly sharshaped, i.e. $\vx \cdot \nu>0$, persists as long as a
solution exists. This property is intertwined with the  demonstration that the
full linearized operator has trivial kernel. \\

Suppose $\Sigma$ is locally represented as the graph of a function
$u \in C^2 (\Omega)$, $u > 0$, in a domain $\Omega \subset \bfR^n$:
$\Sigma = \{(x, u (x)) \in \bfR^{n+1}: \; x \in \Omega\}$,
oriented  by the upward (Euclidean) unit normal vector
field  $\nu$ to $\Sigma$:
\[ \nu = \Big(\frac{-Du}{w}, \frac{1}{w}\Big), \;\; w=\sqrt{1+|Du|^2}. \]
The Euclidean metric 
and second fundamental form of $\Sigma$ are given respectively by
\[ g^e_{ij} = \delta_{ij} + u_i u_j, \,\, h^e_{ij} = \frac{u_{ij}}{w}. \]
According to \cite{CNS4}, the Euclidean principal curvatures
$\kappa^e [\Sigma]$ are the eigenvalues of the symmetric matrix $A^e
[u] = \{a^e_{ij}\}$:
\begin{equation}
\label{eq6.10}
 a^e_{ij} := \frac{1}{w} \gamma^{ik} u_{kl} \gamma^{lj}, \;\;
\gamma^{ij} = \delta_{ij} - \frac{u_i u_j}{w (1 + w)}.
\end{equation}
Note that the matrix $\{\gamma^{ij}\}$ is invertible and equal to
the inverse square root of $\{g^e_{ij}\}$, i.e., $\gamma^{ik}
\gamma^{kj} = (g^e)^{ij}$. By \eqref{eq1.150} the hyperbolic principal
curvatures $\kappa [u]$ of $\Sigma$ are the eigenvalues of the
matrix $A [u] = \{a_{ij} [u]\}$:
\begin{equation}
\label{eq6.20}
 a_{ij} [u] :=  u a^e_{ij} + \frac{\delta_{ij}}{w} =
\frac{1}{w} \Big(\delta_{ij}+ u\gamma^{ik} u_{kl} \gamma^{lj}\Big).
\end{equation}

Problem (\ref{eq1.10})-(\ref{eq1.20}) reduces to the Dirichlet
problem for a fully nonlinear second order equation which we shall
write in the form
\begin{equation}
\label{eq6.140}
G(D^2u, Du, u) = \sigma,
\;\; u > 0 \;\;\; \text{in $\Omega \subset \bfR^n$}
\end{equation}
with the boundary condition 
\begin{equation}
\label{eq6.150}
             u = 0 \;\;\;    \text{on $\partial \Omega$}.
\end{equation}

The function $G$ in equation (\ref{eq6.140}) is determined by
$G (D^2 u, Du, u) = F (A [u])$
where $A [u] = \{a_{ij} [u]\}$ is given by (\ref{eq6.20}).
 Let
  \be \label{eq6.90}
\mathcal{L}=G^{st} \partial_s \partial_t + G^s \partial_s +G_u \ee
be the linearized operator of $G$ at $u$, where
\be \label{eq6.95}
G^{st} = \frac{\partial G}{\partial u_{st}}, \,
 G^s=\frac{\partial G}{\partial u_{s}}, \,
 G_u=\frac{\partial G}{\partial u}.
 \ee
We shall not need the exact formula for $G^s$  but note that
\begin{equation}
\label{eq6.100}
\begin{aligned}
\,& G^{st}=\frac{u}wF^{ij}\gamma^{is}\gamma^{jt}, \;\;
    G^{st}u_{st}=uG_u 
      =G-\frac1w \sum F^{ii}
\end{aligned}
\ee
where $F^{ij} = F^{ij} (A[u])$, etc.
Under condition (\ref{eq1.50}) equation~(\ref{eq6.140}) is elliptic
 for $u$ if  $A[u] \in \mathcal{S}^+$, while (\ref{eq1.60})
implies that $G(D^2 u, Du, u)$ is concave with respect to $D^2 u$.

Since $\vx \cdot \nu=\frac{u-\sum x_k u_k}{w}$, the following lemma is
important.

\begin{lemma}
\label{lem6.1}
We have $\mathcal{L} (u-\sum x_k u_k)=0$.
\end{lemma}

\begin{proof}
Write $\mathcal{L} = L + G_u$. Note that $\mathcal{L} (u_k) = 0$ since horizontal translation is an isometry. We have
\[   \mathcal{L}( x_k u_k)
   = x_k \mathcal{L} (u_k)+u_k L(x_k)+ 2G^{ij} \delta_{ki}u_{kj}
   = u_k G^k + 2 G^{ij} u_{ij} = \mathcal{L} u \]
since $G^{ij} u_{ij} =  u G_u$.
\end{proof}

\begin{lemma}
\label{lem6.2}
Suppose $\mathcal{L} \phi=0 \hspace{.05in} \text{in}\, \,\Omega,\,
\phi=0 \hspace{.05in}\text{on}\,\, \partial \Omega$ and there exists $ v >0$
in $\ol{\Omega}$ satisfying $\mathcal{L}v = 0$. Then $\phi \equiv 0$.
\end{lemma}

\begin{proof}
Set $h=\frac{\phi}{v}$. A simple computation shows that
\[ Lh+2G^{ij}\frac{v_i}{v} h_j =0 \hspace{.05in} \text{in}\, \,\Omega, \;\;
   h=0 \hspace{.05in} \text{on}\,\, \partial \Omega. \]
The lemma now follows by the maximum principle.
\end{proof}

\begin{theorem}
\label{th6.1}
Let $\Omega$ be a strictly starshaped $C^{2,\alpha}$ domain with respect to
the origin.
Suppose $f$ satisfies \eqref{eq1.100} in addition to \eqref{eq1.70}-\eqref{eq1.90}.
There exists a unique solution $u \in C^{\infty} (\ol{\Omega})$ of
the Dirichlet problem
\begin{equation}
\label{eq6.130}
\begin{aligned}
G(D^2 u, Du, u) &=\sigma \;\; \mbox{in $ \Omega$,} \;\;
u = \e \;\;  \mbox{on $\partial \Omega$.}
\end{aligned}
\end{equation}
Moreover, 
the hypersurface $\Sigma = \mbox{graph}(u)$ is strictly starshaped
with respect to the origin. More precisely, there exist constants
$c_0, \e_0 > 0$ such that for all $0 < \e \leq \e_0$,
\begin{equation}
\label{eq6.131}
 \vx \cdot \nu \geq
   \frac{c_0 \nu^{n+1} \sqrt{1 - \sigma^2}}{\sigma}
         \min_{x \in \partial \Omega} x \cdot \vN
\;\; \mbox{on $\Sigma$}
\end{equation}
where $\vN$ is the exterior unit normal to $\partial \Omega$.
\end{theorem}

\begin{proof}
Consider for $0 \leq t \leq 1$, the family of Dirichlet problems
\begin{equation}
\label{eq6.120}
\begin{aligned}
G(D^2 u^t, Du^t, u^t) &=\sigma^t:= t\sigma +(1-t) \;\;
 \mbox{in $ \Omega$,}\\
u^t &= \e \;\; \mbox{on $ \partial \Omega$}.
\end{aligned}
\end{equation}
Starting from $u^0 \equiv \e$ we shall use the continuity method
to prove for any $t \in [0,1]$ that the Dirichlet problem~\eqref{eq6.120}
has a unique solution $u^t \in C^{\infty} (\ol{\Omega})$.
Let $S$ be the set of all such $t$; we know $0 \in S$ so $S$ is not empty.

From the estimates derived in \cite{GSS} and \cite{GS11}  we have
\begin{equation}
\label{eq6.121t}
 |(u^t)^2|_{C^2 (\ol{\Omega})} \leq C \;\; \forall \, t \in S
\end{equation}
where $C$ depends only on $\sigma$ and the exterior ball condition
satisfied by $\Omega$ but is independent of $t$ and $\e$.
This shows that $S$ is a closed set.

Next, let $t \in S$ and denote $w^t = \sqrt{1 + |Du^t|^2}$,
$\vx^t = (x, u^t (x))$.
Then $w^t \vx^t \cdot \nu^t = u^t - \sum x_k u^t_k>0$
and therefore
$\mathcal{L}^t (w^t \vx^t \cdot \nu^t) = 0$ in $\Omega$ by Lemma~\ref{lem6.1}.
Since $\partial \Omega$ is strictly starshaped, by the maximum principle
\begin{equation}
\label{eq6.133}
\begin{aligned}
w^t \vx^t \cdot \nu^t
    \geq \,& \min_{\partial \Omega} w^t \vx^t \cdot \nu^t
       = \min_{\partial \Omega} (u^t-x_k u^t_k)
       = \min_{\partial \Omega} (\e + |\nabla u^t| x \cdot \vN) > \e.
\end{aligned}
\end{equation}
By Lemma \ref{lem6.2}, $\mathcal{L}^t$
has trivial kernel. This shows $S$ is open in $[0,1]$, which is a
standard consequence in elliptic theory of the implicit function
theorem. Therefore $S = [0,1]$, proving the solvability of
the Dirichlet problem~\eqref{eq6.130}.
The uniform starshapeness estimate \eqref{eq6.131} follows from
\eqref{eq6.133} and Lemma~\ref{lem2.20}.
\end{proof}

\begin{proof}[Proof of Theorem~\ref{th1.7}]
Given $f$ satisfying \eqref{eq1.70}-\eqref{eq1.90},
let $f^{\theta}:=(1-\theta)f + \theta H_n^{\frac{1}{n}}$,
$0 < \theta < 1$, which satisfies \eqref{eq1.100} in addition to
\eqref{eq1.70}-\eqref{eq1.90}.
By Theorem~\ref{th6.1}
we obtain a unique solution $u^{\theta, \e} \in C^{\infty} (\ol{\Omega})$
of the approximate problem $f^{\theta} (\kappa [u^\theta]) = \sigma$ with
$u^{\theta, \e}=\e$ on $\partial \Omega$. Moreover, by \eqref{eq6.121t}
\begin{equation}
\label{eq6.121}
 |(u^{\theta, \e})^2|_{C^2 (\ol{\Omega})} \leq C \;\;
\mbox{independent of $\e$}.
\end{equation}
Letting $\e \goto 0$ we obtain a solution $u^{\theta}$ of the asymptotic problem
for $f^{\theta}=\sigma$. By Theorem \ref{th1.5} the principal curvatures of
$\Sigma^{\theta}=\text{graph}(u^{\theta})$ are uniformly bounded by a constant $C$
depending only on $\Omega$ and $\sigma$. Hence as $\theta \goto 0$ we obtain by
passing to a subsequence a smooth locally strictly convex $\Sigma$ satisfying
\eqref{eq1.10}-\eqref{eq1.20} and \eqref{eq6.131}.
\end{proof}

\bigskip

\section{ Uniqueness for mean convex $\Omega$}
\label{sec7}
\setcounter{equation}{0}

In this section we prove Theorem \ref{th1.8}.
We shall assume $\Omega$ is a $C^{2,\alpha}$ domain with
Euclidean mean curvature $\mathcal{H}_{\partial \Omega} \geq 0$.

The main step is to show there is always a solution $\Sigma_2=\text{graph}(u)$ of
the asymptotic problem  \eqref{eq1.10}-\eqref{eq1.20} in $\Omega$ with $G_u<0$ and
moreover $ u\leq v$ for any other solution  $\Sigma_1=\text{graph}(v)$. Then we
show that $\Sigma_2$ is the unique solution.
The proof we give is slightly circuitous in  order to avoid delicate issues of
boundary regularity caused by the degeneracy of the problem at the asymptotic
boundary.

\begin{proposition}
\label{prop7.1}
Let $0 < \sigma < 1$ and $u \in C^2 (\ol{\Omega})$ be a solution of the Dirichlet
problem~\eqref{eq6.130} for $\e > 0$.
Then $G_u <0$ in $\ol{\Omega}$.
Consequently, the linearized operator $\mathcal{L}$ satisfies the maximum
principle and so has trivial kernel.
\end{proposition}

\begin{proof}
Let $\Sigma = \mbox{graph} (u)$ and
$\eta \equiv \frac{\sigma-\nu^{n+1}}u$.
Since $G_u \leq \eta$ by \eqref{eq6.100}, we only need to show
$\eta < 0$ in $\ol{\Omega}$.
According to Theorem \ref{th2.1}, $\eta$ must achieve its maximum
at a boundary point $0\in \partial \Omega$. We choose coordinates so that
the $x_n$ direction is the interior unit normal to $\partial \Omega$ at $0$
where
\be
\label{eq7.10}
\eta_n = \frac{u_n u_{nn}}{u w^3} - \eta \frac{u_n}u < 0, \;
\mbox{or equivalently}, \;  \frac{u_{nn}}{w^3}<\eta.
\ee
On the other hand, by assumptions \eqref{eq1.60} and \eqref{eq1.90},
\[ f (\kappa) \leq \sum f_i ({\bf 1}) \kappa_i = \sum \kappa_i/n. \]
That is the hyperbolic mean curvature $H(\Sigma) \geq \sigma$ and
therefore, equivalently,
\be
 \label{eq7.20}
\frac{1}{w} \Big(\delta_{ij}-\frac{u_i u_j}{w^2}\Big) u_{ij} \geq  n \eta.
\ee
Since $\sum_{\alpha<n} u_{\alpha \alpha}=-u_n (n-1) \mathcal{H}_{\partial D}$,
restricting \eqref{eq7.20} to $\partial \Omega$ implies
\be \label{eq7.30}
\frac{u_{nn}}{w^3}-\frac{u_n}{w}(n-1)\mathcal{H}_{\partial \Omega} \geq n \eta
\ee
Combining \eqref{eq7.10} and \eqref{eq7.30} yields
$w \eta(0) < - u_n \mathcal{H}_{\partial \Omega} \leq 0$.
By Theorem \ref{th2.1} and the maximum principle we obtain
$\eta<0$ in $\ol{\Omega}$.
\end{proof}

\begin{proposition}
\label{prop7.2}
Let $\sigma \in (0,1)$. There exist a solution
$u \in C^{\infty} (\Omega) \cap  C^{0,1} (\ol{\Omega})$ of the Dirchlet problem \eqref{eq6.140}-\eqref{eq6.150}
satisfying $|u^2|_{C^2 (\ol{\Omega})} \leq C$ and $G_u < 0$ in $\Omega$.
\end{proposition}

\begin{proof}
We first assume that $f$ satisfies \eqref{eq1.100} in additon to
 \eqref{eq1.70}-\eqref{eq1.90}.
By an existence theorem in \cite{GSS}, for $\e$ sufficiently small we obtain a
solution $u \in C^{\infty} (\ol{\Omega})$ of the Dirichlet
problem~\eqref{eq6.130}.
By Proposition \ref{prop7.1}, $G_u <0$ in $\ol{\Omega}$.
Therefore the linearized operator at $u$ satisfies the maximum
principle and so has trivial kernel.

By the estimates in \cite{GSS} and \cite{GS11} we have
$|u^2|_{C^2 (\ol{\Omega})} \leq C$ independent of $\e$.
Letting $\e$ tend to $0$ we prove Proposition \ref{prop7.2}
assuming \eqref{eq1.100}.

To remove the assumption \eqref{eq1.100} we consider $f^{\theta}$
in place of $f$ as in the proof of Theorem~\ref{th1.7}.
From the above proof we obtain a solution $u^{\theta}$ of
the asymptotic problem for $f^{\theta}=\sigma$ with $u^{\theta}  = 0$
on $\partial \Omega$.
By Theorem \ref{th1.5} the principal curvatures of
$\Sigma^{\theta}=\text{graph}(u^{\theta})$ are uniformly bounded by a constant $C$
depending only $\partial \Omega$ and $\sigma$.
Let $\theta$ tend to $0$ and note that the condition $G_u \leq 0$ is
preserved in the limiting process and therefore $G_u < 0$ in $\Omega$
by Theorem \ref{th2.1} and the strong maximum principle.
We finish the proof of Proposition \ref{prop7.2}.
\end{proof}

Let $\hu$ denote the solution of \eqref{eq6.140}-\eqref{eq6.150}
constructed in Proposition~\ref{prop7.2}. Theorem~\ref{th1.8} follows
from the following

\begin{proposition}
\label{prop7.3}
Let $v \in C^2 (\Omega) \cap  C^0(\ol{\Omega})$ be a solution of
the Dirchlet problem \eqref{eq6.140}-\eqref{eq6.150}.
Then $v = \hu$.
\end{proposition}

\begin{proof}
We first prove $v \geq \hu$; the strict inequality holds in $\Omega$ unless $v \equiv \hu$.
Let $0 < t \leq 1$, $\epsilon > 0$
and $\Omega_{\epsilon} = \{x \in \Omega: d (x, \partial \Omega) > \epsilon\}$.
For $\epsilon$ sufficiently small, $\partial \Omega_{\epsilon} \in C^{2,\alpha}$ and
$\mathcal{H}_{\partial \Omega_{\epsilon}} \geq 0$.
Applying Proposition~\ref{prop7.2},
let $\hu^{\epsilon, t} \in C^{\infty} (\Omega_{\epsilon})$ be the solution
constructed in Proposition~\ref{prop7.2}
of the Dirichlet problem~\eqref{eq6.140}-\eqref{eq6.150} in $\Omega_{\epsilon}$
with $\sigma$ replaced by $\sigma_t = (1-t) + t \sigma$.
Note that $\sigma_t > \sigma$ and $v > 0 = \hu^{\epsilon, t}$ on $\partial \Omega_{\epsilon}$ for all $0 < t < 1$, and $v > \hu^{\epsilon, t}$ in
$\Omega_{\epsilon}$ for $t$ close to 0. By the maximum principle this property
must continue to hold until $t = 1$.
Thus as $\epsilon \goto 0$ we obtain $v \geq \hu$.  Thus
$v > \hu$ in $\Omega$ or $v \equiv \hu$. \\

Suppose now for contradiction that
\[ \max_{\Omega}(v- \hu)=v(x_0)- \hu(x_0)>0.  \]
Set $w^t:= tv+(1-t) \hu$. We claim that $\text{graph}(w^t)$ is locally strictly convex,
that is, $(w^t)^2+|x-x_0|^2$ is strictly Euclidean convex, in a small neighborhood
of $x_0$.
At $x_0$, $\nabla v =\nabla \hu$ and $D^2 v \leq D^2 \hu$. A simple computation shows
\[w^tw^t_{ij}-tvv_{ij}-(1-t)\hu \hu_{ij}=t(1-t)(v-\hu)(\hu_{ij}-v_{ij}) \geq 0 \hspace{.1in}\mbox{at $x_0$}.\]
Hence at $x_0$,
\[ w^t w^t_{ij}+w^t_i w^t_j +\delta_{ij} \geq
   t(vv_{ij}+v_i v_j+\delta_{ij})+(1-t)(\hu \hu_{ij}+\hu_i \hu_j +\delta_{ij})>0 \]
and the claim follows. So $G(D^2w^t,Dw^t,w^t)$ is well defined near $x_0$. \\

Note that $\frac{d}{dt} G(D^2w^t,Dw^t,w^t) =\mathcal{L}^t w$ near $x_0$ where
$w =v-\hu$. Evaluating at $t=0$ gives
\[\frac{d}{dt} G(D^2w^t,Dw^t,w^t)(x_0)\Big|_{t=0}
     =G^{ij}\Big|_{\hu} w_{ij}(x_0)+G_u\Big|_{\hu} w(x_0) <0. \]
Hence for $t>0$ small enough,
$\varphi(t):=G(D^2w^t,Dw^t,w^t)(x_0)<\sigma$. In particular there is a $t_0 \in (0,1]$ such that
\[\varphi(t_0)=\sigma,\, \varphi(t)<\sigma \hspace{.1in} \mbox{on $(0,t_0)$}.\]

Using the integral form of the mean value theorem, we may write
\[ 0=\varphi(t_0)-\varphi(0)=[a^{ij}w_{ij} +b^s w_s +c(x) w](x_0):=Lw(x_0)+c(x_0)w(x_0)~,\]
where
\[a^{ij} (x) =\int_0^{t_0} G^{ij}\Big|_{w^t} dt,\, \,
  b^s (x) =\int_0^{t_0} G^s\Big|_{w^t} dt,\, \,
  c(x)=\int_0^{t_0} G_u\Big|_{w^t}dt.\]

Since $\text{graph}(w^t)$ is hyperbolic locally strictly convex  in a small neighborhood of $x_0$, the operator
$L=a^{ij}\frac{\partial^2}{\partial x_i \partial x_j}+b^s \frac{\partial}{\partial x_s}$ is elliptic in this neighborhood. Suppose for the moment that also $c(x_0)<0$.
Then $Lw(x_0)=-c(x_0)w(x_0)> 0$ and $w$ has a strict interior maximum at $x_0$ contradicting the maximum principle.

We show $c(x_0)<0$ to complete the proof. According to \eqref{eq6.100},
\[ w^t G_u\Big|_{w^t}(x_0)\leq \varphi(t)- \frac{1}{\sqrt{1 + |D w^t (x_0)|^2}}
   <\sigma - \frac{1}{\sqrt{1 + |D \hu (x_0)|^2}} < 0 \;\;
\mbox{on $(0,t_0)$}. \]
Hence $c(x_0)= \int_0^{t_0} G_u|_{w^t}(x_0)dt <0$.
\end{proof}

\end{document}